\documentclass[12pt]{article}

\usepackage{amsmath,amsthm}
\usepackage{amssymb}

\usepackage{enumitem}

\usepackage{graphicx}

\usepackage[T1]{fontenc}

\newtheorem{theorem}{Theorem}[section]
\newtheorem{corollary}[theorem]{Corollary}
\newtheorem{lemma}[theorem]{Lemma}

\theoremstyle{definition}

\numberwithin{equation}{section}

\usepackage{mathtools}

\DeclarePairedDelimiter\floor{\lfloor}{\rfloor}

\numberwithin{equation}{section}

\makeatletter
\@namedef{subjclassname@2010}{%
  \textup{2010} Mathematics Subject Classification}
\makeatother

\begin{document}
\title{Polynomial Form of Binary Cyclotomic Polynomials}
\author{Aaron Elliot\\
College of Information and Computer Science\\
University of Massachusetts Amherst\\
Computer Science Building, Amherst MA, 01002\\
E-mail: aelliot@cs.umass.edu
}

\maketitle

\begin{abstract}
This paper presents a closed form polynomial expression for the cyclotomic polynomial $\Phi_{pq}(X)$, where $p,q$ are distinct primes. We contrast this against expressions for binary cyclotomic polynomials in (Lam and Leung 1996) and (Lenstra 1979). In addition, we present a related result on factoring $X^{ab}-1$, where $a,b$ are coprime natural numbers.
\end{abstract}

\section{Introduction}

This paper presents two theorems. Theorem \ref{thrm:cycl} is a closed form polynomial expression for the cyclotomic polynomial $\Phi_{pq}(X)$, where $p$ and $q$ are distinct primes. Related to that result, theorem \ref{thrm:factor} is a factorization of $X^{ab} - 1$, where $a$ and $b$ are coprime natural numbers.\\

There exists a rich body of literature exploring the cyclotomic polynomial $\Phi_{pq}(X)$. To my knowledge, the exploration of $\Phi_{pq}(X)$ began with Miggotti in 1883, when he showed that the coefficients of $\Phi_{pq}(X)$ are within $\{1,-1,0\}$.[mIG83]\\

Similarly to this paper, a multitude of authors have proposed explicit expressions for $\Phi_{pq}(X)$. We contrast our expression for $\Phi_{pq}(X)$ against the expressions proposed by Lenstra in 1979[HWL79] and Lam and Leung in 1996[LL96].\\

Lenstra demonstrated in 1979 that
\begin{equation}
    \Phi_{pq}(X) = \sum_{i=0}^{\lambda - 1} \sum_{j=0}^{\mu -1} X^{ip+jq} - \sum_{i=\lambda}^{q-1}\sum_{j=\mu}^{p-1} X^{ip+jq-pq}
\label{eq:Lenstra}
\end{equation}

where $\lambda, \mu\in\mathbb{N}$ such that $0<\lambda<q,$ $0<\mu <p$, $\lambda p \equiv_q 1$ and $\mu q \equiv_p 1$.[HWL79] \\

Lam and Leung demonstrated in 1996 the conditions under which the coefficients of $\Phi_{pq}(X)$ are $1,-1$, or $0$. Using a result from Boot and Greenberg[BG64] that there exists non-negative integers $r,s$ such that $rp+sq = (p-1)(q-1)$, Lam and Leung showed,
\begin{equation}
    \Phi_{pq}(X) = \sum_{k=0}^{(p-1)(q-1)}a_k X^k
\label{eq:LamLeung}
\end{equation}

where $a_k = 1$ if and only if there exists a solution to $k = ip + jq$ for some $i\in[0,r]$ and $j\in[0,s]$; $a_k = -1$ if and only if there exists a solution to $k+pq=ip+jq$ for some $i\in[r+1,q-1]$ and $j\in[s+1,p-1]$; lastly, $a_k=0$ otherwise.[LL96]\\

However, as we will demonstrate in corollary \ref{ch:lamleungLenstra} both (\ref{eq:Lenstra}) and (\ref{eq:LamLeung}) are identical once simplified.\\

The main result of this paper, in theorem \ref{thrm:cycl}, stands at contrast against the expressions (\ref{eq:Lenstra}) and (\ref{eq:LamLeung}). An advantage of the above expressions is that they lend themselves to some analysis tasks well. However, the main disadvantage of both the above expressions is that they require solving for integer solutions of $\mu,\lambda$ and $r,s$ respectively. A possible route for further research is investigating the full repercussions of the expression we present equalling the above expressions for $\Phi_{pq}(X)$.\\

\section{Theorems}

\begin{corollary}
\label{ch:lamleungLenstra}
The expressions (\ref{eq:Lenstra}) and (\ref{eq:LamLeung}) can be simplified to identical expressions.
\end{corollary}
\begin{proof}
From [HWL79], given distinct primes $p$,$q$, and $\lambda, \mu\in\mathbb{N}$ such that $0<\lambda<q,$ $0<\mu <p$, $\lambda p \equiv_q 1$ and $\mu q \equiv_p 1$. However, $\lambda$ is simply the multiplicative inverse of $p$ modulo $q$ in its unique reduced form, and $\mu$ is the multiplicative inverse of $q$ modulo $p$ in its unique reduced form. Euler's Totient Theorem implies that, $\lambda = [p^{q-2}]_q$ and likewise, $\mu = [q^{p-2}]_p$. Implying,

\begin{equation}
    \Phi_{pq}(X) = \sum_{i=0}^{[p^{q-2}]_q - 1} \sum_{j=0}^{[q^{p-2}]_p-1} X^{ip+jq} - \sum_{i=[p^{q-2}]_q}^{q-1}\sum_{j=[q^{p-2}]_p}^{p-1} X^{ip+jq-pq}    
    \label{eq:lenstrareduced}
\end{equation}

From [BG64], as $p,q$ are distinct primes, there exists $r,s$ are such that they are the smallest non-negative integers solution to $rp+sq = (p-1)(q-1)$. As well [BG64] showed that $(s+1)q \equiv_p 1$, and likewise $(r+1)p \equiv_q 1$. By the same logic as above, it follows that, $s = [q^{p-2}]_p- 1$ and $r = [p^{q-2}]_q- 1$.\\

However, [LL96] show that

\begin{equation}
    \Phi_{pq}(X) = \sum_{\substack{k=0 \\ \exists i\in[0,r], j\in[0,s]\\: k = ip + jq }}^{(p-1)(q-1)} X^k - \sum_{\substack{k=0 \\ \exists i\in[r+1,q-1], j\in[s+1,p-1]\\: k+pq = ip + jq }}^{(p-1)(q-1)} X^k
    \label{eq:lamleung2}
\end{equation}
However, $$\{k \in \mathbb{N} : 0\leq k\leq(p-1)(q-1) \text{ and } \exists  i\in[0,r], j\in[0,s] \text{ s.t. } k = ip + jq\}$$ is trivially equal to $\{ip+jq :  i\in[0,r], j\in[0,s]\}$. Likewise,
$$\{k \in \mathbb{N} : 0\leq k\leq(p-1)(q-1) \text{ and } \exists  i\in[r+1,q-1], j\in[s+1,p-1] \text{ s.t. } k+pq = ip + jq\}$$
is trivially equal to  $\{ip+jq-pq :  i\in[r+1,q-1], j\in[s+1,p-1]\}$. Thus, we can rewrite (\ref{eq:lamleung2}) as,

\begin{align}
    \Phi_{pq}(X) &= \sum_{i=0}^{r}\sum_{j=0}^{s}X^{ip+jq} - \sum_{i=r+1}^{q-1}\sum_{j=s+1}^{p-1}X^{ip+jq-pq}\\
    &= \sum_{i=0}^{ [p^{q-2}]_q- 1}\sum_{j=0}^{[q^{p-2}]_p- 1}X^{ip+jq} - \sum_{i= [p^{q-2}]_q}^{q-1}\sum_{j=[q^{p-2}]_p}^{p-1}X^{ip+jq-pq}
    \label{eq:lamleungreduced}
\end{align}

Thus, once the reduced value are substituted for $\mu,\lambda$ and $s,r$  we arrive at the identical formulas (\ref{eq:lenstrareduced}), (\ref{eq:lamleungreduced}).
\end{proof}

\begin{lemma}
Given $a,b,i \in \mathbb{N}$ such that $a$ is coprime to $b$ and $a>b$, $$X^{[ai]_b} + (X^{b}-1)(\sum_{j=1}^{\floor{\frac{ai}{b}}}X^{ai - bj}) = X^{ai} $$
\label{lemma:xai}
\end{lemma}
\begin{proof}
Consider $a,b,i \in \mathbb{N}$ such that $a$ is coprime to $b$ and $a>b$. If i=0, then
$$X^{[a0]_b} + (X^{b}-1)(\sum_{j=1}^{\floor{\frac{a0}{b}}}X^{a0 - bj}) = X^0 + (X^{b}-1)*0 = X^0$$
If i>0, then
\begin{align*}
X^{[ai]_b} + (X^{b}-1)(\sum_{j=1}^{\floor{\frac{ai}{b}}}X^{ai - bj})
= & X^{[ai]_b} + (X^{b}-1)(\sum_{j=1}^{\frac{ai - [ai]_b}{b}}X^{ai - bj})\\
= & X^{[ai]_b} + (\sum_{j=1}^{\frac{ai - [ai]_b}{b}}X^{ai - bj+b})-(\sum_{j=1}^{\frac{ai - [ai]_b}{b}}X^{ai - bj})\\
= & X^{[ai]_b} +(\sum_{j = 0}^{\frac{ai - [ai]_b}{b}-1}X^{ai-bj})-(\sum_{j = 1}^{\frac{ai - [ai]_b}{b}}X^{ai-bj})\\
= & X^{[ai]_b}+ X^{ai} - X^{ai-b(\frac{ai - [ai]_b}{b})}\\
= & X^{ai}\\
\end{align*}
\end{proof}

\begin{theorem}

Given $a,b \in \mathbb{N}$ such that $a$ is coprime to $b$ and $a>b$.
$$X^{ab}-1 = (X^{a}-1)(\sum_{i= 0}^{b-1}X^{i})(1+(X-1)\sum_{i=0}^{b-1}\sum_{j=1}^{\floor{\frac{ai}{b}}}X^{ai - bj})$$
\label{thrm:factor}
\end{theorem}

\begin{proof}

Given $a,b \in \mathbb{N}$ : $a,b$ are coprime and $a>b$,
let $V = \{0,1,2,...,b-1\}$, $W = \{ai : i \in V\}$.\\

From the definition of complete residue systems, [Raj13] definition 13, V forms a complete residue system modulo $b$. As $a,b$ are coprime, from [Raj13] Theorem 24, it follows that W forms a complete residue system modulo $b$. Therefore $\{[i]_b : i\in W\} = \{[i]_b : i\in V\} = V$. Let $ U = \{[i]_b : i\in W\}$. It follows that,
\begin{gather}
\sum_{i=0}^{b-1}(X^{[ai]_b}) = \sum_{u\in U}(X^{u}) = \sum_{v \in V}(X^{v}) = \sum_{i = 0}^{b-1}(X^{i})
\label{eqn:UV}
\end{gather}
Consider $X^{ab} - 1$.
\begin{gather}
X^{ab} - 1 =  (X^{a}-1)\sum_{i=0}^{b-1}X^{ai}
\end{gather}
Using lemma \ref{lemma:xai}, we substitute  $X^{ai}$, 
\begin{gather}
\begin{align}
		 (X^{a}-1)\sum_{i=0}^{b-1}X^{ai}  &=  (X^{a}-1)\sum_{i=0}^{b-1}(X^{[ai]_b} + (X^{b}-1)\sum_{j=1}^{\floor{\frac{ai}{b}}}X^{ai - bj})\\
           &=  (X^{a}-1)(\sum_{i=0}^{b-1}(X^{[ai]_b}) + (X^{b}-1)\sum_{i=0}^{b-1}\sum_{j=1}^{\floor{\frac{ai}{b}}}X^{ai - bj})
\end{align}
\end{gather}
Using equation \ref{eqn:UV},
\begin{gather}
        =  (X^{a}-1)(\sum_{i= 0}^{b-1}X^{i}+(X^b-1)\sum_{i=0}^{b-1}\sum_{j=1}^{\floor{\frac{ai}{b}}}X^{ai - bj})\\
           =  (X^{a}-1)(\sum_{i= 0}^{b-1}X^{i})(1+(X-1)\sum_{i=0}^{b-1}\sum_{j=1}^{\floor{\frac{ai}{b}}}X^{ai - bj})
\end{gather}
Thus,
\begin{gather}
X^{ab}-1 = (X-1)(\sum_{i= 0}^{a-1}X^{i})(\sum_{i= 0}^{b-1}X^{i})(1+(X-1)\sum_{i=0}^{b-1}\sum_{j=1}^{\floor{\frac{ai}{b}}}X^{ai - bj})
\end{gather}
\end{proof}

\begin{theorem} 
Given distinct primes $p$ and $q$ such that $p>q$.
$$\Phi_{pq}(X) = (1+(X-1)\sum_{i=0}^{q-1}\sum_{j=1}^{\floor{\frac{pi}{q}}}X^{pi - qj})$$
\label{thrm:cycl}
\end{theorem}

\begin{proof}
From [Ge] Theorem 4,
$$X^{m}-1   = \prod_{k\in\{k\in \mathbb{N}:k|m\}}\Phi_{k}(X)$$

Given distinct primes $p,q$ such that $p>q$, consider the following:
\begin{gather}
\begin{align}
X^{pq}-1   = & \prod_{k\in\{k\in \mathbb{N}:k|pq\}}\Phi_{k}(X)\\
		   = & \prod_{k\in\{1,p,q,pq\}}\Phi_{k}(X)\\
           = & (X-1)(\sum_{i=0}^{p-1} X^i)(\sum_{i=0}^{q-1} X^i)\Phi_{pq}(X)
\end{align}
\end{gather}

Implying,
\begin{gather}
\Phi_{pq}(X)  =  \frac{X^{pq}-1}{(X-1)(\sum_{i=0}^{p-1} X^i)(\sum_{i=0}^{q-1} X^i)}
\end{gather}

By Theorem \ref{thrm:factor}, this implies,
\begin{gather}
    \Phi_{pq}(X) = \frac{(X-1)(\sum_{i=0}^{p-1}X^i)(\sum_{i= 0}^{q-1}X^{i})(1+(X-1)\sum_{i=0}^{q-1}\sum_{j=1}^{\floor{\frac{pi}{q}}}X^{pi - qj})}{(X-1)(\sum_{i=0}^{p-1} X^i)(\sum_{i=0}^{q-1} X^i)}\\
\Phi_{pq}(X) = 1+(X-1)\sum_{i=0}^{q-1}\sum_{j=1}^{\floor{\frac{pi}{q}}}X^{pi - qj}
\end{gather}

\end{proof}

\section*{References}
\small

[BG64]   J. C. G. Boot and Ralph Greenberg. E1637.The American Math-ematical Monthly, 71(7):799–799, 1964.\\

\noindent[Ge]     Yimin Ge. Elementary properties of cyclotomic polynomials.\\

\noindent[HWL79] Jr H W Lenstra. Vanishing sums of roots of unity. In Proceedings bicentennial congress Wiskundig Genootschap, Math. CentreTracts 100/101, Mathematisch Centrum Amsterdam, page 249,1979.\\

\noindent[LL96]   Tsit-Yeun Lam and Ka Hin Leung. On the cyclotomic polynomial $\phi$pq (x).The American Mathematical Monthly, 103(7):562-564,1996.\\

\noindent[Mig83]  A Migotti.  Zur theorie der kreisteilungs-gleichung.S.-B. der Math.-Naturwiss. Class der Kaiser. Akad. Der Wiss., Wien, 87:7-14, 1883.\\

\noindent[Raj13]   Wissam Raji.An introductory course in elementary number theory. The Saylor Foundation, 2013

\end{document}